\documentclass[preprint,12pt]{elsarticle}
\usepackage{etoolbox}
\makeatletter
\def\ps@pprintTitle{%
	\let\@oddhead\@empty
	\let\@evenhead\@empty
	\def\@oddfoot{\reset@font\hfil\thepage\hfil}
	\let\@evenfoot\@oddfoot
}
\makeatother
\usepackage{bm}
\usepackage{float}
\usepackage{graphics}
\usepackage{amsmath, amsthm, amssymb}
\usepackage{natbib}
\usepackage{caption}
\usepackage[colorlinks=true]{hyperref}
\usepackage{lscape}
\usepackage{float}
\usepackage[utf8]{inputenc} % set input encoding (not needed with XeLaTeX)
\usepackage{amsthm}

\usepackage{enumerate}
\usepackage{mathrsfs}
\usepackage{placeins}
\usepackage[inner=4cm,outer=2cm]{geometry}
\usepackage{graphicx}
\usepackage{amsfonts}
\usepackage{verbatim}
\usepackage{amssymb}
\usepackage{amsthm,multirow}

\newtheorem{thm}{Theorem}[section]

\newtheorem{cor}{Corollary}[section]

\theoremstyle{definition}

\newtheorem{ex}{Example}[section]

\theoremstyle{remark}
\newtheorem{rem}{Remark}[section]

\theoremstyle{properties}

\theoremstyle{Examples}

\numberwithin{equation}{section}

\usepackage[mathscr]{euscript}
\usepackage{geometry} % to change the page dimensions
\usepackage{graphicx,lscape} % support the \includegraphics command and options
\usepackage{booktabs} % for much better looking tables
\usepackage{bigstrut}
\usepackage{array} % for better arrays (eg matrices) in maths
\usepackage{paralist} % very flexible & customisable lists (eg. enumerate/itemize, etc.)
\usepackage{verbatim} % adds environment for commenting out blocks of text & for better verbatim
\usepackage{subfig} 
\biboptions{comma,round,authoryear}
\begin{document}
\begin{frontmatter}
	%\title{\textbf{Nonparametric estimation of\\ quantile-based R\'enyi entropy function}}
	\title{\textbf{Revisiting poverty measures using\\ quantile functions}}
	\author{N. Unnikrishnan Nair, S.M.Sunoj\corref{cor1}}
	\ead{unnikrishnannair4@gmail.com, smsunoj@cusat.ac.in}
	%\author[b]{Silpa Subhash}
	%\ead{silpasubhash112@gmail.com}
	%\author[a]{S.M.Sunoj\corref{cor1}}
	%\ead{smsunoj@gmail.com}
	
	\cortext[cor1]{Corresponding author}
	
	\address{Department of Statistics, Cochin University of Science and Technology, Cochin 682 022, Kerala, India}
	
	%\address[b]{Department of Statistical Sciences, Kannur University, Mangattuparamba Campus 670 567, Kerala, India}
	
%\author{\small N. Unnikrishnan Nair, Silpa Subhash and S. M. Sunoj  }
%\address{Department of Statistics, Cochin University of Science and Technology, Cochin-682022, Kerala, India.}
%	$^{2}$Department of Statistical Sciences,  Kannur University, Mangattuparamba Campus-670567, Kerala, India.}
%\ead{unnikrishnannair4@gmail.com, smsunoj@gmail.com}
%\footnote{$*$Corresponding author. Email address: silpasubhash112@gmail.com} 
	\begin{abstract}
	In this article we redefine various poverty measures in literature in terms of quantile functions instead of distribution functions in the prevailing approach. This enables provision for alternative methodology for poverty measurement and analysis along with some new results that are difficult to obtain in the existing framework. Several flexible quantile function models that can enrich the existing ones are proposed and their utility is demonstrated for real data.	
	\end{abstract}

	
		%\MSC[2010] 94A17 \sep 62G07
	
\end{frontmatter}

%......................................................
\section{Introduction}

\par Poverty, interpreted as a state or condition in which an individual or community lacks resources and other essentials for a minimum standard of living, is a widely discussed subject in socio-economic literature. As a matter of concern for alleviation of poverty through various welfare programmes, there is need to identify the poor and to understand the intensity of poverty. Towards this objective, several measures of poverty such as head count ratio, poverty (income) gap ratio, Gini index for the poor, etc., and their combinations in the form of various indices like the Sen index and the FGT index satisfying certain axioms of behaviour have been proposed in the last few decades. An essential feature of all these indices of poverty is that they are based on the income distribution of the population, despite the recent view that poverty is related to social welfare which includes other facts than income inequality. See \cite{yang2017relationship} for a discussion on the relationship between poverty and income inequality.\\
\par The definition of poverty measures based on income distribution, uses the distribution function of incomes in specifying them. Although probability distributions can also be specified by quantile functions, as an alternative equivalent of distribution functions, the prospect of looking at poverty measures using quantile functions is a rarely visited area in econometrics. The present work is an attempt to redefine the various poverty concepts in the framework of quantile functions and to highlight the additional advantages arising therefrom. In the first place it gives a better insight into the properties of poverty indices by obtaining new results that are difficult to achieve by employing the distribution functions. Many quantile functions do not have tractable distribution functions that limits their practical utility in the existing distribution function based framework.  The fact that they are flexible to contain some of the existing income models as special cases and approximates many others satisfactorily, points out the need to accommodate them also as candidate models in practical problems. In addition to these, generally quantile functions possess properties specific to modelling problems, that are not satisfied by distribution functions. We refer to \cite{gilchrist2000statistical} for details.\\
\par Besides providing the quantile-based definitions of the poverty (income) gap ratio, mean function, Gini index of the poor, Sen and Sen-Shorrocks indices, we show that the first four determines the income law and conditions under which they can be chosen. The relationship between Lorenz curve and the poverty measures are also found. Some highly flexible quantile function that represent the generalized lambda, Wakeby, Kappa and Govindarajulu distributions are proposed as prospective income models and their poverty measures are presented. the utility of our results are demonstrated for income data pertaining to the California state.\\
\par The paper is organized into five sections. After the present introductory section, in Sections 2 and 3 we discuss the additive separable measures and rank based measures of poverty. In Section 4 we propose some flexible quantile functions as income models. Finally, the applications of our results in analysing data on incomes of the California state are presented in Section 5.

\section{Additively separable measures}

Let $X$ be a continuous non-negative random variable representing the income in a population with distribution function $F(x)$, probability density function $f(x)$ and quantile function
$$
Q(u)=\inf [x \mid F(x) \geq u], \quad 0 \leq u \leq 1.
$$
When $F(x)$ is continuous and strictly increasing, $F(Q(u))=u$, so that on differentiation
$$
f(Q(u))=[q(u)]^{-1}
$$
where $q(u)=\frac{d Q(u)}{d u}$ is the quantile density function and $F(x)=u$ if and only if $\quad x=Q(u)$.\\
\par We assume that $x=t$ is the poverty line so that all individuals with income below $t$ are considered poor and that there exist a function representing the poverty of the individual. A general approach (\cite{atkinson1987measurement}) to the definition of a poverty measure is to propose an additively separable aggregate poverty metric

\begin{equation}\label{2.1}
A_F(t)=\int_0^t a(x, t) d F(x)
\end{equation}
with $a(x, t)$ interpreted as a deprivation suffered with income $x$ satisfying $a(x, t)=0$ for $x \geq t$. Many poverty measures discussed in literature like those of \cite{watts1969economic}, \cite{thon1979measuring}, \cite{clark1981indices}, etc., belong to this class. When expressed in  terms of quantile functions \eqref{2.1} has the form

\begin{equation}\label{2.2}
A_Q(u)=\int_0^u a(Q(p), Q(u)) d u.
\end{equation}

An important class of measures suggested by \citet{foster1984class} reads

\begin{equation}\label{2.3}
A_F(t, \alpha)=\int_0^t\left(1-\frac{x}{t}\right)^\alpha d F(x)
\end{equation}
or equivalently

\begin{eqnarray}\label{2.4}
A_Q(Q(u), d)=A_Q(u, \alpha)=\int_0^u\left(1-\frac{Q(p)}{Q(u)}\right)^\alpha dp.
\end{eqnarray}
When $\alpha=0,$ \eqref{2.3} reduces to $F(t)$ which is the head count ratio giving the proportion of poor in the population. Another important special case, discussed in many contexts is obtained when $\alpha=1$,

\begin{equation}\label{2.5}
 A_Q(u, 1) = A_1(u) = u - \frac{1}{Q(u)} \int_0^u Q(p) d p 
\end{equation}
called the poverty gap ratio (PGR). It is interpreted as the shortfall of the income of the poor from the poverty line and indicates the poverty line and the average income of the poor. Some new properties of PGR and their implications to poverty analysis are presented below.

\begin{thm}\label{tm1}
	The income distribution is uniquely determined by $A(u)$ through the formula
\begin{equation}\label{2.6}
Q(u)=\frac{\exp \left[-\int_u^1 \left(p - A_1(p)\right)^{-1}d p \right]}{u-A_1(u)}.
\end{equation}
\end{thm}
\begin{proof}
 From \eqref{2.5},
	$$
	\frac{Q(u)}{\int_0^u Q(p) d p}  =(u-A_1(u))^{-1} $$
 or
	$$\frac{d}{d u} \log \int_0^u Q(p) d p  =(u-A_1(u))^{-1} $$
giving
	$$\int_0^u Q(p) d p  =\exp \left[-\int_u^1(p-A_1(p))^{-1} d p\right].
	$$
Differentiating the last equation we have \eqref{2.6}.
\end{proof}	 
	
Some implications of Theorem \ref{tm1} are the following
\begin{enumerate}
	\item Theorem \ref{tm1} establishes a one-to-one correspondence between the PGR and the income distribution, by saying that the functional form of $A_1(u)$ determines $Q(u)$. Thus if the functional form of $A_1(u)$ is known or obtained by considering the empirical form of $A_1(u)$, then the income distribution can be found directly from \eqref{2.6}(see Section 5). The expression of $A_1(u)$ for several quantile functions are exhibited in Table 1 for easy reference to the corresponding distribution.
	\item  A large number of distributions prescribed in literature as models of income, were not based on a clear rationale, but simply justified by the fact they provided satisfactory goodness of fit. In modelling, it is a widely accepted fact that some distinct features of data must form the basis of selecting the model. Characterization theorems have a great role in finalising the appropriate model. Taking these into consideration by fixing appropriate PGR for the data we have a criterion to choose the model as well as to give sufficient representation for empirical features observed in the data. As in the case of parametric Lorenz curves one can fix suitable parametric PGR for the data. Since any function of $t$ cannot be a PGR, we fix the conditions for an arbitrary function $A_1: R \rightarrow[0,1]$ to be a PGR.
\end{enumerate}
\begin{thm}\label{tm2.2}
A function $A_1: R \rightarrow[0,1]$ will be the poverty gap ratio relating to a distribution $F(x)$, if and only if it satisfies the following properties
\begin{enumerate}[(a)]
	\item $\lim _{t \rightarrow \infty} \frac{d}{d t} t A_1(t)=1 ; \lim _{t \rightarrow 0} \frac{d}{d t} t A_1(t) \geq 0$
	\item $A_1(t)$ is increasing and differentiable
	\item $t A_1(t)$ is convex
	\item $ A_1(t)=\int_0^t \frac{t-x}{t} d F(x), \quad A_1(t) = A_F(t, 1).$	
\end{enumerate} 
\end{thm}
 \begin{proof}
 	Let $A_1(t)$ be as in (d). Then $$
 	t A_1(t)  =\int_0^t(t-x) d F(x)=\int_0^t F(x) d x $$
 	giving
 	$$
 	F(t)  =\frac{d}{d t} t A_1(t)	$$
 	
 	It is not difficult to see that $F(x)$ satisfies all the conditions of a distribution function of which $A_1(t)$ is the PGR by definition. Conversely if $F(x)$ is the distribution function with PGR; $A_1(t)$, then properties (i) to (iv) hold as seen from the above.
 \end{proof} 
\begin{rem}
The corresponding properties of $A_1(u)$ can be obtained as in Theorem \ref{tm2.2}. We require $A_1 (u)$ to satisfy the conditions that
\begin{enumerate}[i)]
	\item $Q(u)=\frac{\exp \left[-\int_u^1\left(p-A_1(p)\right)^{-1} d p\right]}{u-A_1(u)}$
	\item $A_1(u)$ is an increasing function for $0 < u < 1$ and
	\item  $A_1(u)$ is continuous.	
\end{enumerate} 
\end{rem}
\begin{ex}
Let $A_1(u)=\frac{u^2}{\beta}, \quad 0 \leq u \leq 1, \quad \beta>1.
$  Then from \eqref{2.6} by direct calculation,
$$
Q(u)=\frac{\beta(\beta-1)}{(\beta-u)^2}.
$$
Inverting, $X$ has distribution function
$$
F(x)=\beta-\left(\frac{\beta(\beta-1)}{x}\right)^{\frac{1}{2}}, \quad \frac{\beta-1}{\beta} \leq x \leq \frac{\beta}{\beta-1}.
$$

The Lorenz curve of this distribution is discussed in \cite{rohde2009alternative}. One can read from Table 1, the distributions specified by quantile functions that are generated by special forms of $A_1(u)$.	
\end{ex}

Sometimes, the distribution is represented by its quantile density function $q(u)$ and $Q(u)$ may not have a simple enough expression to use \eqref{2.4}. We observe that
$$
A_1(t)=\frac{1}{t} \int_0^t F(x) d x,
$$
so that when $t=Q(u)$,
\begin{equation}\label{2.7}
A_1(u)=\frac{1}{Q(u)} \int_0^u p q(p) d p.
\end{equation}

A well known case is the distribution specified by
$$
q(u)=k u^{\alpha}(1-u)^{\beta}, k>0, 
$$
$\alpha, \beta$ real which contains the exponential, Pareto II, re-scaled beta, log logistic, Govindarajulu distributions as special cases and approximates many continuous distributions. In this case
$$
A_1(u)=\frac{B_u(\alpha+2, \beta +1)}{B_u(\alpha+1, \beta+1)}, \alpha, \beta>-2.
$$
where $B_u(\alpha, \beta)=\int_0^u p^{\alpha-1}(1-p)^{\beta-1} d p, \alpha, \beta>0$, is the incomplete beta function. An alternative inversion formula for $Q(u)$ is
\begin{equation}\label{2.8}
Q(u)=\exp \left[-\int_u^1 \frac{A_1^{\prime}(p)}{p-A_1(p)} d p\right].
\end{equation}
\begin{ex}
	Assume $x$ to follow power distribution $F(x)=\left(\frac{x}{\alpha}\right)^{\beta}, 0 \leq x \leq 1$ or $Q(u)=\alpha u^{1/\beta}$. Then for this distribution
	$$
	A_1(u)=\frac{u}{\beta+1},
	$$
	 proportional to the head count ratio.
	
\end{ex}
A third special case of \eqref{2.3} is

\begin{equation}\label{2.9}
A_2(t)=A_F(t, 2)=\int_0^t\left(1-\frac{x}{t}\right)^2 d F(x)
\end{equation}
which measures depth of poverty while $A_F(t, 1)$ indicates its
severity. Corresponding to $A_F(t, 2)$, we have

\begin{equation}\label{2.10}
A_2(u)=A_Q(u, 2)=\int_0^u\left(1-\frac{Q(p)}{Q(u)}\right)^2 d p.
\end{equation}
\begin{thm}\label{tm2.3}
The function $A_2(t)$ can be determined from $A_1(t)$ as

\begin{equation}\label{2.11}
A_2(t)=\frac{2}{t^2} \int_0^x x A_1(x) d x.
\end{equation}
	
\end{thm}
\begin{proof}
Since $ \frac{1}{t} \int_0^t F(x) d x=A_1(t)$, we have $F(t)=\frac{d}{d t} t A_1(t)$.
Hence
$$
\begin{aligned}
t^2 A_2(t) & =\int_0^t(t-x)^2 d F(x) \\
& =2 \int_0^t(t-x) d F(x) \\
& =2 \int_0^t(t-x)\left(\frac{d}{d x} x A_1(x)\right) d x \\
& =2 \int_0^t x A_1(x) d x,
\end{aligned}
$$
from which \eqref{2.11} follows.	
\end{proof}
\begin{ex}
For the power distribution $Q(u)=\alpha u^{1 / \beta}$,
$$
A_2(u)=\int_0^u\left(1-\frac{p^{\frac{1}{ \beta}}}{u^ {\frac{1}{ \beta}}} \right)^2 d p=\frac{2 u}{(\beta+1)(\beta+2)}.
$$
Also observe that
$$
A_2(u)=\frac{2}{\beta+2}A_1(u).
$$
	
\end{ex}
\begin{rem}
We see that $A_2(u)$ and $A_1(u)$ are proportional. From \eqref{2.11} it is not difficult to show that $A_2(u)=C A_1(u)$, when $C$ is a constant only $X$ has power distribution.
\end{rem}

 A somewhat similar poverty index was proposed by \cite{clark1981indices} as

\begin{equation}\label{2.12}
C_\beta(t, \beta)=\frac{1}{\beta} \int_0^t\left[1-\left(\frac{x}{t}\right)^{\beta}\right] d F(x), \beta<1 .
\end{equation}
For $\beta=0$, it reduces to the Watt's measure to be discussed below and for $\beta=1$, it is the poverty gap ratio. The index \eqref{2.12} is the \cite{chakravarty1983ethically} metric when ${\beta>0}$. We have the quantile equivalent of \eqref{2.12} as $C_{\beta}(u, \beta) = \frac{1}{\beta} \int_0^u\left[1-\left(\frac{Q(p)}{Q(u)}\right)^{\beta}\right] d p$.  When $\beta=2$,
$$
C_2(u)=C_2(u, 2)=\frac{1}{2} \int_0^u\left(1-\left(\frac{Q(p)}{Q(u)}\right)^2\right) d u.
$$
\begin{thm}\label{tm2.4}
The function $C_2(u)$ is determined form $C_1(u)=A_1(u)$
\end{thm}
\begin{proof}
That $C_1(u)=A_1(u)$ is mentioned above.
Now,
$$
\begin{aligned}
A_2(u) & =\int_0^u\left(1-\frac{Q(p)}{Q(u)}\right)^2 d p \\
& =u-\frac{2}{Q(u)} \int_0^u Q(p) d p+\frac{1}{Q^2(u)} \int_0^u Q^2(p) d p
\end{aligned}
$$
giving
$$
\begin{aligned}
& \frac{1}{Q^2(u)} \int_0^u Q^2(p) d p=A_2(u)+\frac{2}{Q(u)} \int_0^u Q(p) d p-u \\
& =A_2(u)+2\left(u-A_1(u)\right)-u \\
&
\end{aligned}
$$
Hence
$$
\begin{aligned}
C_2(u) & =\frac{1}{2 Q^2(u)}\left[\int_0^{\infty}\left[Q^2(u)-Q^2(p)\right] d p\right. 
\end{aligned}
$$
\begin{equation}\label{2.13}
=\frac{u}{2}- \frac{1}{2 Q^2(u)} \int_0^u Q^2(p) d p
\end{equation}
From \eqref{2.12} and \eqref{2.13}
$$
C_2(u)=A_1(u)-\frac{1}{2} A_2(u)
$$
Using Theorem \ref{tm2.3}, $A_2(u)$ is determined from $A_1(u)$ and $A_1(u)=C_1(u)$ and Theorem \ref{tm2.4} is proved.	
\end{proof}

When $\beta$ tends to zero we have the Watt's measure \citep{watts1969economic}
$$
C_0(t, 0)=W_F(t)=\int_0^t(\log z-\log x) d F(x)
$$
or equivalently
$$
 \begin{aligned}
W_Q(u) & =\int_0^u[\log Q(u)-\log Q(p)] d p 
\end{aligned}
$$
\begin{equation}\label{2.14}
 =u \log Q(u)-\int_0^u \log Q(p) d p.
\end{equation}
Like other measures $W_Q(u)$ also determines the income distribution by the inversion formula
\begin{equation}\label{2.15}
Q(u)=\exp \left[-\int_u^1\frac{W'(p)}{p} dp\right].
\end{equation}
An example of income model with interesting properties is the power distribution.
\begin{thm}\label{tm2.5}
The Watt's measure $W_Q(u)=\frac{u}{\beta}$, if and only if $X$ has power distribution $Q(u)=\alpha u^{\frac{1}{\beta}}$.	
\end{thm}

Note that in this case $W( t)$ is proportional to the head count ratio.\\
\par Since poverty is to a large extent influenced by the inequalities of income in the population, it is natural that measures of inequality are linked to poverty indices. A major tool in the evaluation of income inequality in the Lorenz curve (LC)

\begin{equation}\label{2.16}
L(u)=\frac{1}{\mu} \int_0^u Q(p) d p
\end{equation}
where $\mu=E(X)$ is assumed to be finite. It follows that $L(u)$ determines $Q(u)$ up bo a scale transformation, since
$$
Q(u)=\mu \frac{d L(u)}{d u}.
$$
From \eqref{2.16} and \eqref{2.5}
$$
A_1(u)=u-\frac{L(u)}{L^{\prime}(u)}.
$$
Finally, the Gini index is
$$
\begin{aligned}
& G=1-2 \int_0^1 L(p) d p \\
& =1-2 \int_0^1\left(\exp -\int_u^1(p-A(p))^{-1} d p\right) d u \text {, } \\
&
\end{aligned}
$$
in terms of PGR.  \\

Generally, the LC is estimated by interpolation or directly from the distribution function of $X$ or through parametric forms. A large number of parametric forms of LC have been proposed in literature. Some important forms, their PGR's and their quantile functions are exhibited in Table \ref{table1} for easy reference.\\
\begin{table}
	\begin{center}
		\begin{tabular}{ccc}
		\hline
		Lorenz curve	& PGR &  $Q(u)$\\ 	\hline
			
			$u\exp [-\delta(1-u)]$& $\frac{u^2\delta}{1+\delta u}$ & $\mu (u\delta+1)e^{-\delta}(1-u)$ \\
			\cite{kakwani1973estimation}&& \\ \hline
		$	u(1-\delta(1-u)^{\frac{1}{2}}), 0<\delta<1$& $\frac{u^2}{(1-u)^{\frac{1}{2}}+u\delta-\delta}$  & $\mu [1-\delta (1-u)^{\frac{1}{2}}+u\delta(1-u)^{-\frac{1}{2}}]$ \\\cite{kakwani1980class}& &\\ \hline
		
			$\frac{e^{ku}-1}{e^k-1}$& $u-\frac{1}{k}(1-e^{-ku})$ & $\frac{\mu e^{ku}}{e^k-1}$  \\ \cite{chotikapanich1993comparison}& &\\ \hline
		
		$\frac{(1-\theta)^2u}{(1+\theta)^2-4u}$	& $\frac{4\theta u^2}{(1+\theta)^2}$ & $\frac{\mu(1-\theta^2)^2}{((1+\theta)^2-4\theta u)^2}$ \\
		\cite{aggarwal1984optimum}&&\\  \hline
		$uT^{u-1}, T>1$	& $\frac{u^2 \log T}{1+u \log T}$ & $\mu [1+\mu \log T]T^{u-1}$  \\
		\cite{gupta1984functional}&&\\	\hline
		$u[1-(1-u)^{\theta}], 0 <\theta \leq 1$	& $\frac{\theta u^2 (1-u)}{1+u \theta (1-u)^{\theta-1}-(1-u)^{\theta}}$ &  $\mu [1+(1-u)^{\theta-1}((\theta+1)u-1)]$\\
		\cite{ortega1991new}&&\\  \hline
		$\frac{u(\beta-1)}{\beta-u}, \beta>1$& $\frac{u^2}{\beta}$& $\frac{\mu \beta (\beta-1)}{(\beta-u)^2}$\\
		\cite{rohde2009alternative}&&\\
			\hline
		\end{tabular}
	\end{center}
\caption{Functions associated with parametric Lorenz curves}
\label{table1}
\end{table}
\par As a weakness of parametric LC's it is pointed out that they provide distribution functions that are difficult to manipulate algebraically to extract distributional properties. However a glance at the above table shows that the quantile functions have quite simple forms that admit analytic treatment.
\section{Rank-based poverty measures}
\par A second category of poverty metrics called rank-based measures that use a poor person's rank within the poor or whole population as an indicator of relative deprivation have the structural form
\begin{equation}\label{3.1}
P(F,t)=\int_0^t q(F,x;z)dF(x).
\end{equation}
Before discussing them, we need some basic concepts that are constituents of \eqref{3.1}. The first measure is the income gap ratio (IGR) for the poor defined as (\cite{sen1986gini})
\[
I-F(t)=1-\frac{\int_0^t xdF(x)}{tF(t)}
\]
or equivalently,
\begin{equation}\label{3.2}
I_Q(u)=1-\frac{\int_0^uQ(p)dp}{Q(u)}=\frac{1}{uQ(u)}\int_0^u pq(p)dp.
\end{equation}
It gives the average shortfall of incomes among the poor community. We also need the average income of the poor
\[
\mu_F(t)=\frac{1}{F(t)}\int_0^txdF(x)=E(x|X\leq t)
\]
and its quantile form
\begin{equation}\label{3.3}
\mu_Q(u)=\frac{1}{u}\int_0^u Q(p)dp=\frac{1}{u}\int_0^u pq(p)dp
\end{equation}
and the Gini index of the poor
\[
G_F(t)=1-\frac{2}{\mu_F(t)}\int_0^t x\left(1-\frac{F(x)}{F(t)}\right)\frac{f(x)}{F(t)}dx
\]
or
\[
G_Q(t)=1-\frac{2}{u^2\mu_Q(u)}\int_0^u \left(u-p\right)Q(p)dp
\]
\begin{equation}\label{3.4}
=\frac{2}{u^2\mu_Q(u)}\int_0^u pQ(p)dp-1.
\end{equation}
Some of the properties of the functios $I_Q$, $\mu_Q$ and $G_Q$ discussed in \cite{nair2008some} and \cite{nair2022cumulative} that are needed in the sequel are the following
\begin{enumerate}[(a)]
	\item the function $D(u)$ determines the income distribution as 
	\[Q(u)=\frac{\mu}{D(u)}\exp \left[-\int_u^1 D(p)dp\right]\]
	where
	\[
	D(u)=u(1-I_Q(u)).
	\]
	For the power distribution $Q(u)=\alpha u^{\frac{1}{\beta}}, I_Q(u)=(1+\beta)^{-1},$ a constant.
	\item $G_Q(u)$ is independent of the poverty quantile $u$ if and only if $X$ has power distribution, where $G_Q(u)=(1+2\beta)^{-1}.$
	\item The above measures satisfy the identity with Lorenz curve 
	\begin{itemize}
		\item [(i)] $L(u)=\mu^{-1}u \mu_Q(u)$
		\item [(ii)] $L(u)=\mu^{-1}D(u)Q(u), I_Q(u)=1-\frac{\mu L(u)}{uQ(u)}$
		\item [(iii)] $G_Q(u)=1-\frac{2}{uL(u)}\int_0^u L(p)dp$
		and $L(u)=\frac{2}{u(1-G-Q)}\exp \left[-\int_0^u\frac{2 dp}{p(1-G_Q(p))}\right].$
	\end{itemize}
It is often difficult to find the distribution function $F(t)$ from $G_F(t)$. An advantage of the quantile approach initiated here is that the distribution $Q(u)$ can be expressed in terms of $G_Q(u)$.
\end{enumerate}
\begin{thm}
	The income distribution specified by $Q(u)$ is given by
	\begin{equation}\label{3.8}
	Q(u)=B(u) \exp \left[-\int_u^1 B(p) d p\right]
	\end{equation}
	where $B(u)=\frac{1+G_Q(p)+p G_Q^{\prime}(p)}{p\left(G_Q(p)-1\right)}.$
\end{thm}
\begin{proof}
From \eqref{3.4},
$$
\frac{1-G(u)}{2} \mu_Q(u)=\int_0^u\left(\frac{1}{u}-\frac{p}{u^2}\right) Q(p) d p.
$$
Using \eqref{3.2},
$$
\frac{1-G(u)}{2} \int_0^u Q(p) d p=\int_0^u Q(p) d p-\frac{1}{u} \int_0^u p Q(p) d p
$$
which simplifies to
$$
\frac{1+G(u)}{2}  u  \int_0^u Q(p) d p=\int_0^u p Q(p) d p
$$
Differentiating with respect to $u$, we get after some algebra,
$$
u[(G(u)+1)-2 u] Q(u)=\left[u G^{\prime}(u)+G(u)+1\right] \int_0^u Q(p) d p
$$
and
$$ \frac{Q(u)}{\int_0^u Q(p) d p}=B(u) $$
or
$$ \frac{d}{d u} \log \left(\int_0^u Q(p) d p\right)=B(u)$$
and
\begin{equation}\label{3.9}
\int_0^u Q(p) d p=\exp \left[-\int_u^1 B(p) d p\right]
\end{equation}
The proof is completed by noting that \eqref{3.8} is the derivative of the last expression \eqref{3.9}.	
\end{proof} 
\begin{cor}
All the basic functions associated with poverty can be expressed in terms of the PGR, $A_1(u)$.
\begin{enumerate}[(i)]
	\item the IGR, $ I(u)=\frac{A_1(u)}{u}$
	\item mean income of the poor, $\mu_Q(u)=\frac{1-A_1(u)}{u} Q(u)$
	\item Lorenz curve $L(u)=\mu^{-1} u\left(u-A_1(u)\right)$
	\item Gini index of the same
	$$
	\begin{array}{r}
	G_Q(u)=\frac{\int_0^u p D(p) \exp \left[-\int_p^1 D(r) d r\right] d p }{ \int_0^u D(p) \exp \left[-\int_p^1 D(r) d r\right] d p}.
	\end{array}
	$$
\end{enumerate}	
\end{cor}

The basic rank-based poverty measure proposed by \cite{sen1976poverty} is
$$
S_F(t)=H(t)\left[I_F(t)+\left(1-I_F(t)\right) G_F(t)\right]
$$
where $H(t)=F(t)$ is the head count ratio. We can rewrite it as
\begin{equation}\label{3.10}
S_Q(u)=S_F(Q(u))=u\left[I_Q(u)+\left(1-I_Q(u)\right) G_Q(u)\right].
\end{equation}
Since the basic formula \eqref{3.10} needs calculation of several constituents, it is useful to have a single expression for $S_Q(u)$ in terms of $Q(u)$. Substituting for $I_Q$,
$$
S_Q(u)=u\left[1-\frac{\int_0^u Q(p) d p}{u Q(u)}+\frac{\int_0^u Q(p) d p}{u Q(u)} G_Q(u)\right]
$$
giving
$$
\begin{aligned}
u-S_Q(u) & =\frac{\int_0^u Q(p) d p}{Q(u)}\left[1-G_Q(u)\right] \\
& =2 \frac{\int_0^u Q(p) d p}{Q(u)}\left[1-\frac{1}{u \mu_Q^{(u)}} \int_0^u p Q(p) d p\right] \\
& =2 \frac{\int_0^u Q(p) d p}{Q(u)}\left[1-\frac{\int_0^u p Q(p) d p}{u \int_0^u Q(p) d p}\right]
\end{aligned}
$$
Thus
$$
\left(u-S_Q(u)\right) u Q(u)=2 u \int_0^u Q(p) d p-\int_0^u p Q(p) d p
$$
leading to the formula
\begin{equation}\label{3.11}
S_Q(u)=u-\frac{2\int_0^uQ(p)dp}{Q(u)}+\frac{2\int_0^upQ(p)dp}{uQ(u)}.
\end{equation}
Conversely,
$$
\frac{d}{d u} u\left(u-S_Q\right) Q(u)=2 \int_0^u Q(p) d p
$$
giving a second order differential equation with variable coefficients
\begin{equation}\label{3.12}
u\left(u-S_Q\right) \frac{d^2 y}{d u^2}+\left(2 u-S_Q-u S_Q^{\prime}\right) \frac{d y}{d u}-2 y=0
\end{equation}
with $y=y(u)=\int_0^u Q(p) d p$. It seems that a general expression for $Q(u)$ in terms of $S_Q(u) $is  difficult to find. However it is useful in the case of some specific distributions. For example when $S(u)=K u$, where $K$ is a positive constant \eqref{3.12} gives
$Q(u)=\alpha u^{1 / \beta}$, the power distribution. There are two solutions of \eqref{3.12} of which one is the power distribution and the other is a term
that cannot be a quantile function so that $S(u)=Ku$ characterizes the power distribution. Some distributions and their expressions for $S(u)$ are given below
\begin{itemize}
	\item [(i)] power distribution, $Q(u)=\alpha u^{1 / \beta}: S(u)=\frac{3 \beta+1}{(\beta+1)(2 \beta+1)} u$
	\item [(ii)] exponential, $Q(u)=-\frac{1}{\lambda} \log (1-u)$ :
	$$
	S(u)=2-\frac{1}{u}-\frac{3 u-2}{2 \log (1-u)}
	$$
	\item [(iii)] Pareto I: $\quad Q(u)=\sigma(1-u)^{-\frac{1}{\alpha}} ; \quad F(x)=1-\left(\frac{\sigma}{x}\right)^\alpha, x>\sigma>0$
	$$
	S(u)=u+\frac{2 \alpha}{1-\alpha}(1-u)^{\frac{1}{\alpha}}+\frac{2 \alpha^2}{(1-\alpha)(1-2 \alpha)} u\left[(1-u)^2-(1-u)^{\frac{1}{\alpha}}\right]
	$$
	\item [(iv)] Govindarajuln: $Q(u)=\sigma\left[\left(\beta_1\right) u^{\beta} \beta_u-\beta u^{ \beta+1}\right]$.
	$$
	S(u)=u-\frac{2 u}{\beta+1}+2 u^2\left(\frac{\beta+2)(\beta+3)-\beta(\beta+2) u}{(\beta+2)(\beta+3)-(\beta+1-\beta u)}\right.
	$$
	\item [(v)] Dagum $Q(u) = b\left[u^{-\frac{1}{\beta}}-1\right]^{-\frac{1}{a}} ; \quad F(x)=\left[\left(1+\left(\frac{b}{x}\right)^a\right]^{-\beta}\right.$
	$$
	S(u)=u-\frac{2 b}{a Q(u)}\left[B_{\left(\frac{1}{b}\right)^a}\left(\beta+\frac{1}{a}, 1-\frac{1}{a}\right)-\frac{1}{u} B_{\left(\frac{1}{b}\right)^a}\left(2 \beta+\frac{1}{a}, 1-\frac{1}{a}\right)\right.
	$$
\end{itemize}
\begin{figure}[H]
	\centering
	\includegraphics[width=0.7\linewidth]{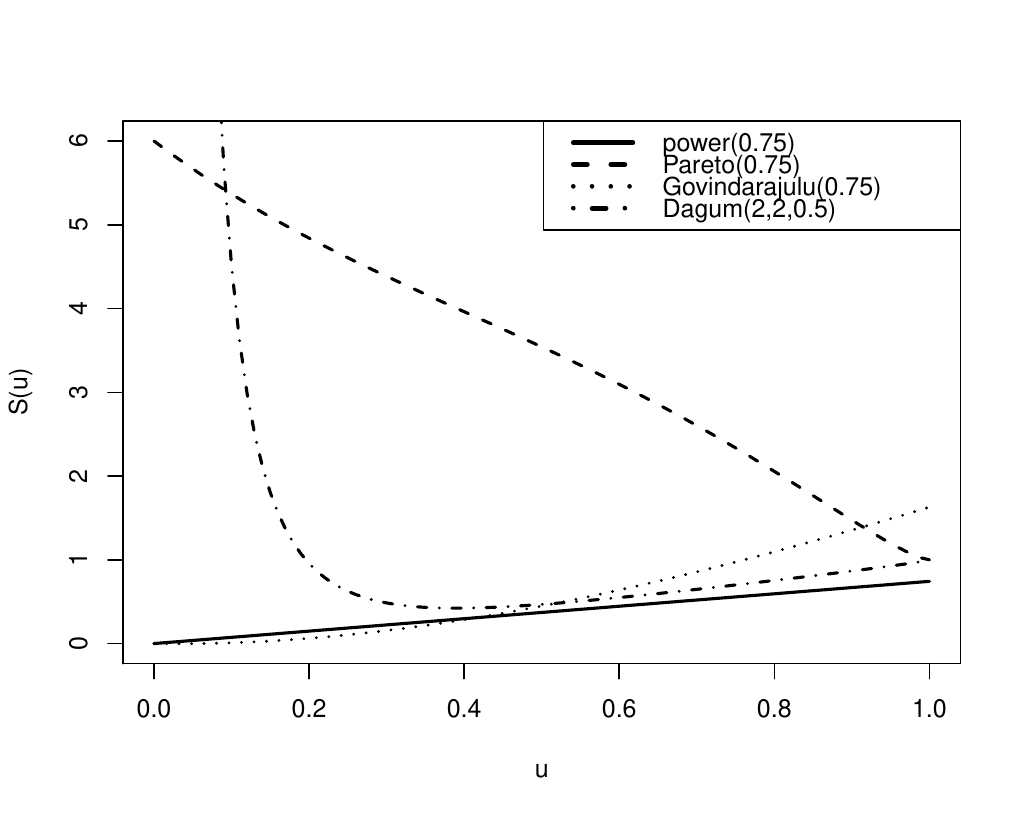}
	\caption{$S(u)$ of power, Pareto, Govindarajulu and Dagum distribution with different parameter values.}
	\label{fig:su}
\end{figure}
In Figure \ref{fig:su} the graphs of the index for the power distribution, Pareto I, Govindarajulu and Dagum distributions are shown. Based on the graph and other analysis, we see that
\begin{enumerate}[(i)]
	\item The index of the power distribution is proportional to the head count ratio and it is the only distribution possessing this property
	\item $S(u)$ can be increasing (power), decreasing (Pareto I) or initially decreasing and then increasing (Dagum) in $u$ (head count ratio). Since $u = F(t)$ is increasing (decreasing) means $t$ is also increasing (decreasing) and vice-versa the monotonicity of $S_Q(u)$ and $S_F(t)$ are the same.
	\item  When an income model is chosen arbitrarily from a list of candidates to study poverty using $S_Q (u)$, it should be compatible with the monotonic nature of the Sen index.
\end{enumerate} 

While Sen index expressed in the general form \eqref{3.1} has $q(F, x ; z) = $ $2\left[1-\frac{F(x)}{r(z)}\right]\left(1-\frac{x}{z}\right)$, where $r_z$ is the proportion below the poverty line, other prominent indices proposed in literature are
\begin{enumerate}[(a)]
	\item \cite{takayama1979poverty} with $q(F, x ; z)=2(1-F(x))\left(1-\frac{x}{\mu_F}\right)$ where $\mu_F=r_z \mu_p+\left(1-r_z\right) z$ and $\mu_p$ is the mean income of the poor.
	\item \cite{kakwani1980class}: $q(F, x ; z)=(k+1)\left(1-F(x) / r_z\right)^k\left(1-\frac{x}{z}\right)$
	and 
	\item \cite{thon1983poverty}: $q(F, x ;z)= \frac{2}{c-1}\left[ c-zF(x) \right] \left( 1-\frac{x}{z}\right), c \geq 2$
	\item \cite{foster1984class}: $F(s)=H\left(I^2(t)+(1-I)^2 C_p^2(t)\right]$ where $H$ is the head count ratio and $C_p$ is the coefficient of variation among the poor. These can also be treated similar to $S(u)$.
\end{enumerate}
\par \cite{shorrocks1995revisiting} provided a modification of the Sen index which satisfies the properties of symmetry, replication invariance, monotonic and homogeneous of degree zero in incomes and poverty line. See \cite{hagenaars2017definition}, \cite{xu2020sen} and \cite{chakravarty2019shorrocks} for details. The Shorrocks-Sen index, as it is called, takes the form
$$
S S_F(x)=\mu(x)[1+G(x)]
$$
where $\mu$ is the mean and $G$ is the Gini index of $X=\frac{t-Y}{t}$ and $Y$ is the income. In terms of $F(t), I_F(t)$ and $G_F(t)$ we have,
$$
S S_F(t)=(2-F(t)) F(t) I_F(t)+F^2(t)\left(1-I_F(t)\right) G_F(t).
$$

It is easy to see that the quantile version is

\begin{equation}\label{3.13}
S S_Q(u)=(2-u) u I_Q(u)+u^2\left(1-I_Q(u)\right) G_Q(u) .
\end{equation}
The form gives properties different from the Sen poverty index, for example,
In the power distribution, $ S S_Q(u)$ is not proportional to the head count ratio, but

$$
S S(u)=\frac{u(2+4 \beta-(\beta+1) u)}{(\beta+1)(2 \beta+1)}
$$
is a quadratic function. The derivative of $SS(u)$ is
$$
\frac{d S S}{d u}=\frac{(2+4 \beta)-2 u(\beta+1)}{(\beta+1)(2 \beta+1)}>0 \text { for all } u \text {. }
$$
In the interval $[0,1], S S(u)$ is always an increasing function of $u$.\\
\par A third approach to poverty measure by \cite{hagenaars1987class} visualizes it in the tradition of Dalton with the aid of utility functions with the general definition
$$
P_F( t)=A(t) \int_0^t[U(t) - U(x)] d F(x)
$$
where $U(.)$ is the individual's utility function and $A(t)$ is a normalization factor. This is equivalent to
\[
P_Q(u)=P_F(Q(u))=A_H(u)\int_0^u\left[g(u)-g(p)\right]dp
\]
where $A_H(u)=A(Q(u))$ and $g(p)=U(Q(p))$. This class has some common measures like those of \cite{watts1969economic}, \cite{chakravarty1983ethically} and \cite{clark1981indices} disused earlier as particular cases.
\section{Income modelling by quantile functions}

\par The statistical regularities observed in income date have prompted researchers even from the early days to promote parametric models to represent them. In many cases the data generating mechanism was satisfactorily approximated by distribution functions and the summary measures derived from them were quite useful in indicating income characteristics. However, the absence of any concrete rationale for the choice of a specific model resulted in prescribing a large number of distributions as income models for data separated by time and space.  A few important distributions specified by their quantile functions that have potential as income models are given below along with the expressions of poverty measures pertaining to them.
\begin{enumerate}[(i)]
	\item Kappa distribution introduced by \cite{hosking1994four} has quantile function

\begin{equation}\label{4.1}
	Q_k(u)=\alpha+\frac{\beta}{\gamma}\left[1-\left(\frac{1-u^\delta}{\delta}\right)^\gamma\right]
\end{equation}
where $\alpha$ is a location parameter and $\beta$ is a scale parameter. Since the random variable $X$ has to be non-negative, $Q(0) \geq 0$, giving $\alpha+\frac{\beta}{\gamma}\left(1-\delta^{-\gamma}\right) \geq 0$ along with $\beta>0$. The advantage of \eqref{4.1} is that it contains as special cases, the Pareto distribution for $\delta=1$, Stoppa for $\gamma<0, \delta>0$ and $\beta=-\gamma(\delta)^\gamma, \alpha=\frac{\beta}{\gamma} ;$ Dagum, when $\gamma<0, \delta<0$, $\beta=\gamma(\delta)^\gamma$ and $\alpha=\frac{\beta}{\gamma}$, Burr type III $(\gamma<0, \delta<0)$, Burr Type $XII$  $(\gamma>0, \delta<0)$, generalized Pareto $(\delta=1)$, exponential $\gamma=1, \delta=0$. It is also obtained as the distribution of $\frac{1}{Y}$, where $Y$ is the generalized Weibull of \cite{mudholkar1994generalized}. Besides the distribution \eqref{4.1} gives good approximation to other unimodal distributions $\gamma, \delta<1$.  \cite{tarsitano2006new} has used \eqref{4.1} to analyse the income data related to the net disposable in Italian household budget for the years 1993, 1995, 1998 and 2000. Various expressions for poverty measures are\\

\quad poverty gap ratio, $A_1(u)=u-\frac{\beta P_1(u ; \gamma, \delta)}{\gamma \delta Q(u)}-\frac{\alpha u}{Q(u)}$ \\

\quad Lorenz curve, $L(u)=\frac{1}{\mu}\left[\alpha u+\frac{\beta}{\gamma \delta} P_1(u, \gamma, \delta)\right]$,\\

\quad mean income, $\mu_{Q_k}(u)=\frac{1}{u}\left[\alpha u+\frac{\beta}{\gamma \delta} P_1(u, \gamma, \delta)\right]$,\\

\quad income gap rato, $I_{Q_k}(u)=1^{r \delta} \frac{1}{u Q(u)}\left[\alpha u+\frac{\beta}{\gamma \delta} P_1(u, r, \delta)\right]$, and the \\

\quad Gini coefficient of the poor
$$G_{Q_k}(u)=\frac{2}{u^2\mu_Q}\left[\frac{\alpha u^2}{2}+\frac{\beta}{\gamma \delta}P_2(u,\gamma, \delta)\right],
$$
where,
$$
\begin{aligned}
& P_1(u, \gamma, \delta)=\int(1-t)^\gamma(1-t \delta)^{\frac{1}{\delta}-1} d t \\
& P_2(u, \gamma, \delta)=\int(1-t)^\gamma(1-t \delta)^{\frac{1}{s}-1}
\end{aligned}
$$

with the integrals evaluated over $\left(\frac{1-u^\delta}{\delta}, \frac{1}{\delta}\right)$.
\item  Generalized lambda distributions:  The form proposed by \cite{ramberg1972approximate} has quantile function
\begin{equation}\label{4.2}
Q_R(u)=\lambda_1+\lambda_2^{-1}\left[u^{\lambda_3}-(1-u)^{\lambda_4}\right]
\end{equation}

where $\lambda_2>0$ and $\lambda_1-\frac{1}{\lambda_2} \geq 0$ for non-negativity. The general properties of the model and regions of validity together with conditions that approximate \eqref{4.2} to many continuous distribution functions used for income analysis are given in \cite{karian2000fitting}. \cite{tarsitano2004fitting} has employed this distribution as an income model. For this model, we have
$$
\begin{aligned}
& A_1(u)=u-\frac{1}{Q_R(u)} \lambda_1+\lambda_2^{-1}\left(\frac{u^{\lambda_3+1}}{\lambda_3+1}+P_3(u)\right),
\end{aligned}
$$
\[
L(u)=\frac{1}{\mu}\left(\lambda_1u+P_3(u)\right),\; \mu=\lambda_1+\lambda_2^{-1}\left(\frac{1}{\lambda_3+1}-\frac{1}{\lambda_4+1}\right),
\]
\[
\mu_{Q_R}(u)=\frac{1}{\mu} \left( \lambda_1u+P_3(u)\right),
\]
\[
I_{Q_R}(u)=1-\frac{1}{uQ_R(u)}\left(\lambda_1u+P_3(u)\right),
\]
where $P_3(u)=\frac{u^{\lambda_3+1}}{\lambda_3+1}+\frac{(1-u)^{\lambda_4+1}}{\lambda_4+1}$
and $$
G_{Q_R}(u)=u^2 \frac{2}{\mu_Q(u)}\left[\frac{\lambda_1 u^2}{2}+\frac{u^{\lambda_3+2}}{\lambda_2\left(\lambda_3+2\right)}+\frac{u(1-u)^{\lambda_4+1}}{\lambda_2\left(\lambda_4+1\right)}+\frac{(1-u)^{\lambda_4+1}-1}{\lambda_2\left(\lambda_4+2\right)}\right].
$$

An alternative version proposed by Freimer et
al. $(2008)$ is

$$
\begin{aligned}
 Q_F=\lambda_1+\lambda_2^{-1}\left(\frac{u^{\lambda_3}-1}{\lambda_3}-\frac{(1-u)^{\lambda_4}-1}{\lambda_4}\right)
\end{aligned}
$$

where $\lambda_1-\frac{1}{\lambda_2 \lambda_3} \geq 0$. \cite{haridas2008modelling} have discussed the role of this distribution for modelling income data. It contains the power, Pareto I \& II distributions and approximates the Weibull, Singh-Maddala, Dagum and Fisk distributions for different ranges of parameter values. We have 
$$
\begin{aligned}
& A_1(u)=u-\frac{1}{Q(u)}P_4(u)
\end{aligned}
$$
\[
L(u)=\frac{1}{\mu}\left(P_4(u)\right),
\]
\[
\mu_{Q_F}(u)=\frac{1}{\mu} \left( \frac{1}{\lambda_2\lambda_4+\lambda_1}+P_4(u)+\frac{(1-u)^{\lambda_{4+1}-1}}{\lambda_2\lambda_4(\lambda_{4}+1)}\right),
\]
\[
I_{Q_F}(u)=1-\frac{1}{uQ(u)}\left(\frac{1}{\lambda_2\lambda_4}+P_4(u)+\frac{(1-u)^{\lambda_{4+1}-1}}{\lambda_2\lambda_4(\lambda_{4}+1)}\right),
\]
where $P_4(u)=\frac{1}{\lambda_2\lambda_4}+\lambda_1+\frac{u}{\lambda_2\lambda_4}+\frac{u^{\lambda_3+1}}{\lambda_2\lambda_3(\lambda_{4}+1)}+\frac{(1-u)^{\lambda_4+1}-1}{\lambda_2\lambda_4(\lambda_3+1)}$
and $$
G_{Q_F}(u)=\frac{2}{u^2\mu_Q}\left[\left( \lambda_1 -\frac{1}{\lambda_2\lambda_3}+\frac{1}{\lambda_2\lambda_4}\right)\frac{u^2}{2}+\frac{u^{\lambda_3+2}}{\lambda_2\lambda_3\left(\lambda_3+2\right)}+\frac{u(1-u)^{\lambda_4+1}}{\lambda_4+1}+\right.$$ $$\left. \frac{(1-u)^{\lambda_4+1}-1}{(\lambda_4+1)\left(\lambda_4+2\right)}\right].
$$

\item  Wakeby distribution: Introduced by \cite{houghton1978birth}, the quantile function is given by

\begin{equation}\label{4.3}
Q_w(u)=\lambda+\theta(1-u)^{-\alpha}-\phi(1-u)^\beta
\end{equation}

which has five parameters of which $\lambda$ specifies location, $\theta$ and $\phi$ the scales and $\alpha$ and $\beta$ the shapes. For the random variable to be non-negative we should have $\lambda+\theta-\phi \geq 0$ and $\theta, \phi>0$.
An important property of the model is that it can satisfactory represent positively skewed data like the log-normal, gamma and beta laws. It contains the Pareto I and Pareto II distributions as special cases. The poverty measures are
$$
\begin{aligned}
& A_1(u)=u-\frac{1}{Q_w(u)} \quad P_5(u), \\
& L(u)=\frac{1}{\mu} P_5(u), \\
& \begin{array}{l}
\mu_{Q_W}(u)=\frac{1}{u} P_{5}(u), \\
I_{Q_W}=1-\frac{P_5(u)}{u Q_w(u)},
\end{array}
\end{aligned}
$$
where $\quad P_5(u)=\lambda u+\frac{\theta(1-u)^{-\alpha+1}-1}{\alpha-1}+\phi \frac{(1-u)^{\beta+1}-1}{\beta+1}, \alpha>1, \beta>-1$
and
$$ G_Q(u)=\frac{2}{u^2 \mu_{Q_w}(u)}\left[\frac{\lambda u^2}{2}+\frac{\theta}{\alpha-1}\left\{u(1-u)^{-\alpha+1}+\frac{(1-u)^{-\alpha+2}-1}{\alpha-2}\right\}\right. +$$ 
$$\left.  \frac{\phi}{\beta+1}\left\{u(1-u)^{\beta+1}+\frac{(1-u)^{\beta+2}-1}{\beta+2}\right\}\right].
$$
\item Govindarajulu distribution: This distribution is defined by
$$
Q(u)=\sigma\left((\beta+1) u^\beta-\beta u^{\beta+1}\right), \sigma,  \beta>0.
$$
Its properties and application are described in \cite{nair2013quantile}.
In this case
$$
\begin{aligned}
A_1(u) & =\frac{\beta(\beta+2-\beta u-u)}{(\beta+2)(\beta+1-\beta u)}, \\
L(u) & =\frac{1}{2}(\beta+2-\beta u) u^{\beta+1}, \\
\mu_Q(u) & =\sigma \frac{(\beta+2-\beta u)}{\beta+2} u^\beta, \\
I_Q(u) & =\frac{\beta}{\beta+1-\beta u},
\end{aligned}
$$
and
$$
G_Q(u)=\frac{2(\beta+2)}{\beta+2-\beta_2 u}\left(\frac{\beta+1}{\beta+2}-\frac{\beta u}{\beta+3}\right).
$$
Other flexible quantile functions that have a similar role to play in income analysis can be seen in \cite{nair2013quantile}.
\end{enumerate}

\begin{rem}
All the distributions except (i) given above do not have tractable distribution functions and hence the conventional methods cannot be applied to them. 	
\end{rem}
\begin{rem}
	In the above formulas setting $u=1$ in $\mu_Q(u)$ gives the mean of the distribution and similarly $u=1$ in $G_Q(u)$ leads to the usual Gini index.
\end{rem}

\section{Applications to data analysis}
In this section we demonstrate the usefulness of the results obtained above in the analysis of poverty and income 
inequality based on real data. The per capita personal incomes in dollars of 58 countries in the state of California in U.S for the year 2019. The data in ascending order of magnitude is given in Table \ref{table2}.\\
\begin{table}
	\begin{center}
		\begin{tabular}{cccccccc}
			
			38130 & 38445 & 39443 & 40447 & 41077 & 41267 & 41843 & 42043 \\
			
			42418 & 42845 & 43268 & 43471 & 43536 & 44259 & 45487 & 45742 \\
			
			45920 & 47139 & 47245 & 47245 & 47605 & 47860 & 48438 & 48841 \\
			
			49194 & 49654 & 51088 & 51131 & 51342 & 52976 & 53500 & 53505 \\
			
			54715 & 55261 & 55266 & 55910 & 56123 & 56534 & 59838 & 60513 \\
			
			61004 & 63542 & 63729 & 65094 & 66076 & 66700 & 68936 & 69898 \\
			
			71592 & 71711 & 72155 & 75717 & 81171 & 85324 & 11547 & 134107 \\
			
			139405 & 147135 &  &  &  &  &  &  \\			
		\end{tabular}
	\end{center}
\caption{Per capita income of 58 countries in California}
\label{table2}
\end{table}

\par Let $x_{1 : n}, x_{2: n}, \ldots x_{n:n}$ be the ordered incomes with $x_{0: n}=0$. Then $F(x_{i:n})=\frac{i}{n},\; i=1,2,...,n$ and the empirical quantile function $\bar{Q}\left(u_i\right)=x_{i: n}, \frac{i-1}{n} \leqslant u_i \leqslant \frac{i}{n}$ with $u_0=0$. At the values of $u$ in $(0,1), \bar{Q}(u)$ by linear interpolation at the plot position $\frac{i-\frac{1}{2}}{n}$, enable $\bar{Q}(u)$ to be differentiable and the derivative $\bar{q}(u)=\frac{d \bar{Q}}{d u}$,
becomes $\bar{q}_i(u)=n\left(x_{i+1: n}-x_{i: n}\right), \frac{i-\frac{1}{2}}{n}<u<\frac{i+\frac{1}{2}}{n}$. Then it is known that $\bar{q}(u)$ is a 
consistent estimator of $q(u)$. Accordingly, we propose consistent estimators for the functions $\mu_Q, I_Q$ and $A_1 (u)$ directly from the sample
\begin{enumerate}[(a)]
	\item $\bar{\mu}_Q\left(u_i\right)=\frac{1}{i} \sum_{i=1}^{j-1} i\left(x_{i+1, n}-x_{i: n}\right), j=\left[n u_i\right]$ the greatest integer contained in $n u_i$.
	for $\mu_Q(u)=\frac{1}{u} \int_0^u p q(p) d p$.
	\item $\bar{I}_Q\left(u_i\right)=\sum_{i=1}^{j-1}\frac{ i\left(x_{i+1: n}-x_{i: n}\right)}{ix_{j:n}}=x_{j:n} \bar{\mu}_Q\left(u_i\right)$
	for $ I_Q\left(u_i\right)  =\frac{1}{uQ(u)} \int_0^u p q(p) d p $ and
	\item $\bar{A}_1\left(u_i\right) = u_i \bar{I}_Q\left(u_i\right)$
	for $A_1(u)=\frac{1}{Q(u)} \int_0^u p q(p) d p$.
\end{enumerate} 
Using the data on incomes in Table \ref{table2} and the above estimators, the estimates of $\mu_Q(u)$ at the data points are given in Table \ref{table3}.
\begin{table}[H]
	\begin{center}
		\begin{tabular}{cccccc}
			$u$ & $\mu_Q(u)$ & $u$& $\mu_Q(u)$ & $u$& $\mu_Q(u)$ \\ \hline
			0.0345 & 19.7315 & 0.1207 & 1328.5588 & 0.2069 & 1911.2304 \\
			0.0577 & 304.2465 & 0.1379 & 1179.2204 & 0.2241 & 1876.7298 \\
			0.0690 & 755.8168 & 0.1552 & 1533.8510 & 0.2414 & 2206.1195 \\
			0.0862 & 1382.4415 & 0.1724 & 1484.0924 & 0.2586 & 2813.1330 \\
			0.1034 & 1495.4265 & 0.1897 & 1530.1792 & 0.2759 & 3604.7200 \\
			0.2931 & 4143.5181 & 0.5690 & 9230.4205 & 0.8276 & 19508.5318 \\
			0.3103 & 4232.9025 & 0.5862 & 9234.1131 & 0.8448 & 19630.5264 \\
			0.3276 & 4270.1685 & 0.6034 & 9168.6294 & 0.8621 & 20287.0419 \\
			0.3498 & 4130.3180 & 0.6207 & 9079.784 & 0.8793 & 22534.0065 \\
			0.3621 & 3985.7430 & 0.6379 & 8968.4554 & 0.8966 & 24025.0041 \\
			0.3793 & 4160.8242 & 0.6552 & 9622.0868 & 0.9138 & 26250.7014 \\
			0.3906 & 4087.9917 & 0.6724 & 12254.8224 & 0.9310 & 30938.4824 \\
			0.4138 & 4476.6540 & 0.6847 & 12834.8073 & 0.9483 & 58220.4540 \\
			0.4310 & 4732.0262 & 0.7069 & 14872.7752 & 0.9655 & 75757.0443 \\
			0.4483 & 4873.7952 & 0.7241 & 15707.5824 &  &  \\
			0.4655 & 5675.5410 & 0.7414 & 15734.6901 &  &  \\
			0.4828 & 5503.8662 & 0.7586 & 15759.2574 &  &  \\
			0.5170 & 6680.2736 & 0.7759 & 15855.2400 &  &  \\
			0.5345 & 7094.1000 & 0.7931 & 15774.5500 &  &  \\
			0.5517 & 7665.6070 & 0.8103 & 19349.9520 &  &  \\ \hline
		\end{tabular}
	\end{center}
	\caption{ Estimates of mean income of the poor}
	\label{table3}
\end{table}
\begin{figure}[H]
	\centering
	\includegraphics[width=0.8\linewidth]{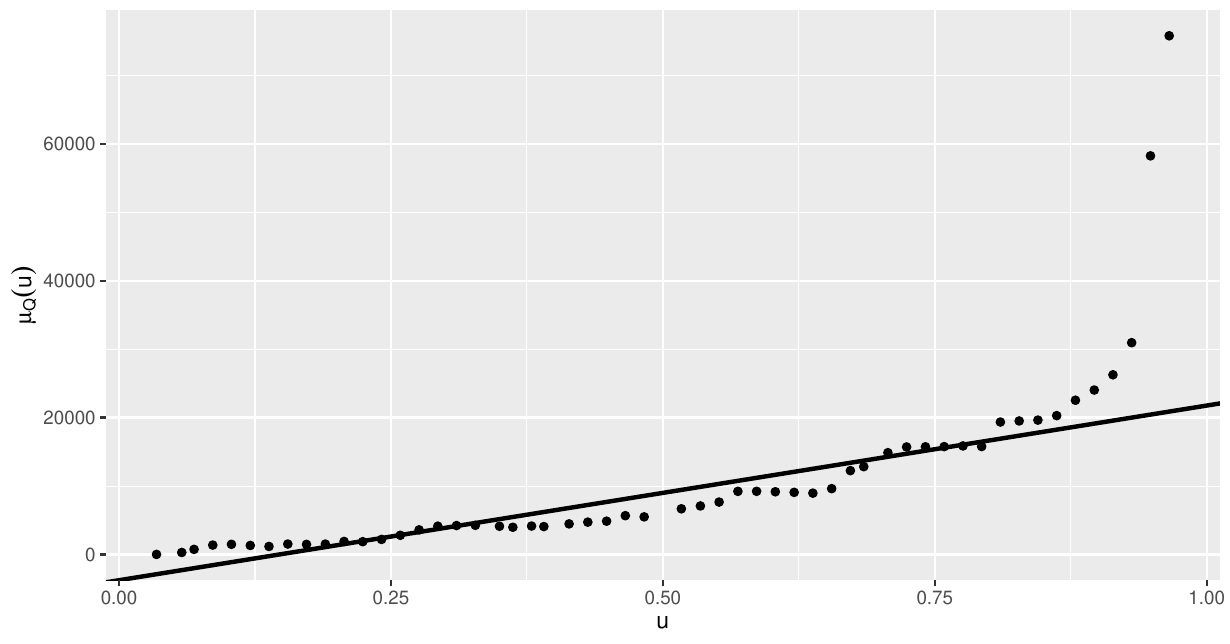}
	\caption{Mean income of the poor.}
	\label{fig:umuQ}
\end{figure}
\begin{table}[h]
	\centering
	\begin{tabular}{cccccc}
		$u$ & $G_Q(u)$ & $u$ & $G_Q(u)$ & $u$ & $G_Q(u)$ \\
		\midrule
		0.0345 & 0.0000000 & 0.1207 & 0.2336292 & 0.2069 & 0.8809066 \\
		0.0517 & 1.8720359 & 0.1379 & 1.9753794 & 0.2241 & 1.4186406 \\
		0.0690 & 1.8946424 & 0.1552 & 1.6473920 & 0.2414 & 1.6535262 \\
		0.0862 & 1.8770444 & 0.1724 & 2.0753379 & 0.2586 & 1.8274994 \\
		0.1034 & 1.6277299 & 0.1897 & 2.2754333 & 0.2759 & 3.1668855 \\
		0.2931 & 2.7828750 & 0.5690 & 2.1078872 & 0.8276 & 1.9962204 \\
		0.3103 & 1.8186444 & 0.5862 & 1.4990645 & 0.8448 & 1.2676027 \\
		0.3276 & 1.2639615 & 0.6034 & 1.3377579 & 0.8621 & 1.2150751 \\
		0.3498 & 1.2564311 & 0.6207 & 1.2898577 & 0.8793 & 1.5178803 \\
		0.3621 & 1.4415978 & 0.6379 & 1.3856341 & 0.8966 & 1.6428995 \\
		0.3793 & 1.7089461 & 0.6552 & 1.5809869 & 0.9310 & 1.6808132 \\
		0.3906 & 1.7493353 & 0.6724 & 1.6180994 & 0.9483 & 1.7363457 \\
		0.4138 & 1.5560488 & 0.6847 & 1.5423056 & 0.9655 & 1.8902764 \\
		0.4310 & 1.5361280 & 0.7069 & 1.9358923 &  &  \\
		0.4483 & 2.0041658 & 0.7241 & 1.9797122 &  &  \\
		0.4655 & 2.0575689 & 0.7414 & 2.0218067 &  &  \\
		0.4828 & 2.1521697 & 0.7586 & 2.2553937 &  &  \\
		0.5170 & 2.2471304 & 0.7759 & 2.5766037 &  &  \\
		0.5395 & 2.6172275 & 0.7931 & 2.8351479 &  &  \\
		0.5517 & 2.7180134 & 0.8103 & 2.7097428 &  &  \\ 
		\bottomrule
	\end{tabular}
	\caption{Estimates of Gini index for the poor}
	\label{table4}
\end{table}
\begin{figure}[H]
	\centering
	\includegraphics[width=0.8\linewidth]{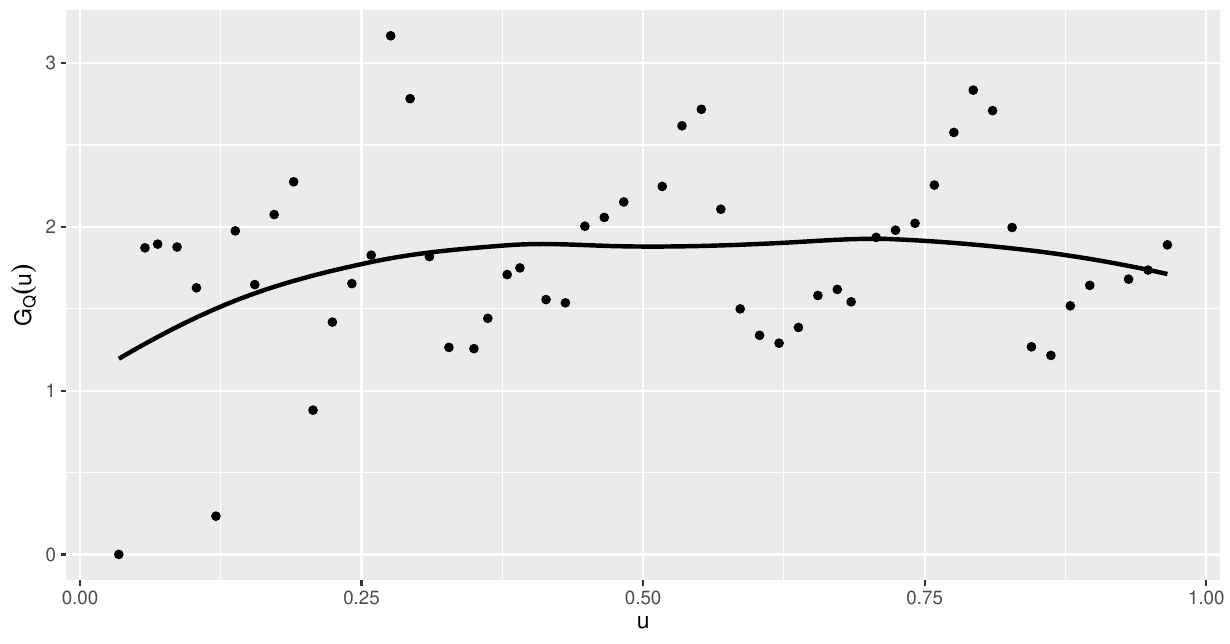}
	\caption{Gini index plot for the poor}
	\label{fig:uGQu}
\end{figure}
The graph of $\left(u, \mu_Q(u)\right)$ exhibited in Figure \ref{fig:umuQ} (after omitting the last four values being inconsistently large to represent the poverty line) gives evidence that $\mu_Q(u)$ is approximately a straight line. Hence based on this empirical fact, we take
$$
\mu_Q(u)=\alpha+\beta u, \quad 0<u<1.
$$
Hence from 
$$
\mu_Q(u)=\frac{1}{u} \int_0^u p q(p) d p
$$
or
$(\alpha+\beta u) u=\int_0^u p q(p) d p$.  Differentiating and simplifying,
$$
q(u)=\frac{\alpha}{u}+2 \beta
$$
or \begin{equation}\label{5.1}
Q(u)=\int_0^u q(p) d p=\alpha \log u+2 \beta u+r, \quad 0<u<1,\; \alpha, \beta, r>0
\end{equation} 
\begin{table}[h!]
	\centering
	\begin{tabular}{cccccc}
		$u$ & $S(u)$ & $u$ & $S(u)$ & $u$ & $S(u)$ \\ \hline
		0.0345 & 0.0005 & 0.0862 & 0.1420 & 0.1724 & 0.3087 \\
		0.0517 & 0.0190 & 0.1034 & 0.1775 & 0.1897 & 0.3520 \\
		0.0690 & 0.0552 & 0.1207 & 0.1896 & 0.2069 & 0.4829 \\
		0.1379 & 0.1946 & 0.2241 & 0.5064 & 0.2414 & 0.6305 \\
		0.1552 & 0.2704 & 0.2586 & 0.8610 & 0.2759 & 1.1775 \\
		0.2931 & 1.4064 & 0.3103 & 1.5215 & 0.3276 & 1.6146 \\
		0.3498 & 1.6397 & 0.3621 & 1.6600 & 0.3793 & 1.8039 \\
		0.3906 & 1.8414 & 0.4138 & 2.0930 & 0.4310 & 2.2872 \\
		0.4483 & 2.3850 & 0.4655 & 2.8860 & 0.4828 & 2.8944 \\
		0.5170 & 3.5308 & 0.5395 & 3.8454 & 0.5517 & 4.2420 \\
		0.5690 & 5.2297 & 0.5862 & 5.3472 & 0.6034 & 5.5143 \\
		0.6207 & 5.5216 & 0.6379 & 5.5930 & 0.6552 & 6.1272 \\
		0.6724 & 7.5776 & 0.6847 & 8.0598 & 0.7069 & 9.5082 \\
		0.7241 & 9.8880 & 0.7414 & 10.1229 & 0.7586 & 10.1682 \\
		0.7759 & 10.3200 & 0.7931 & 10.4060 & 0.8103 & 12.6900 \\
		0.8276 & 12.6132 & 0.8448 & 12.8874 & 0.8621 & 13.5792 \\
		0.8793 & 15.3027 & 0.8966 & 18.1300 & 0.9310 & 18.4926 \\ \hline
	\end{tabular}
	\caption{Estimates of Sen index}
	\label{table5}
\end{table}
\begin{figure}[h!]
	\centering
	\includegraphics[width=0.8\linewidth]{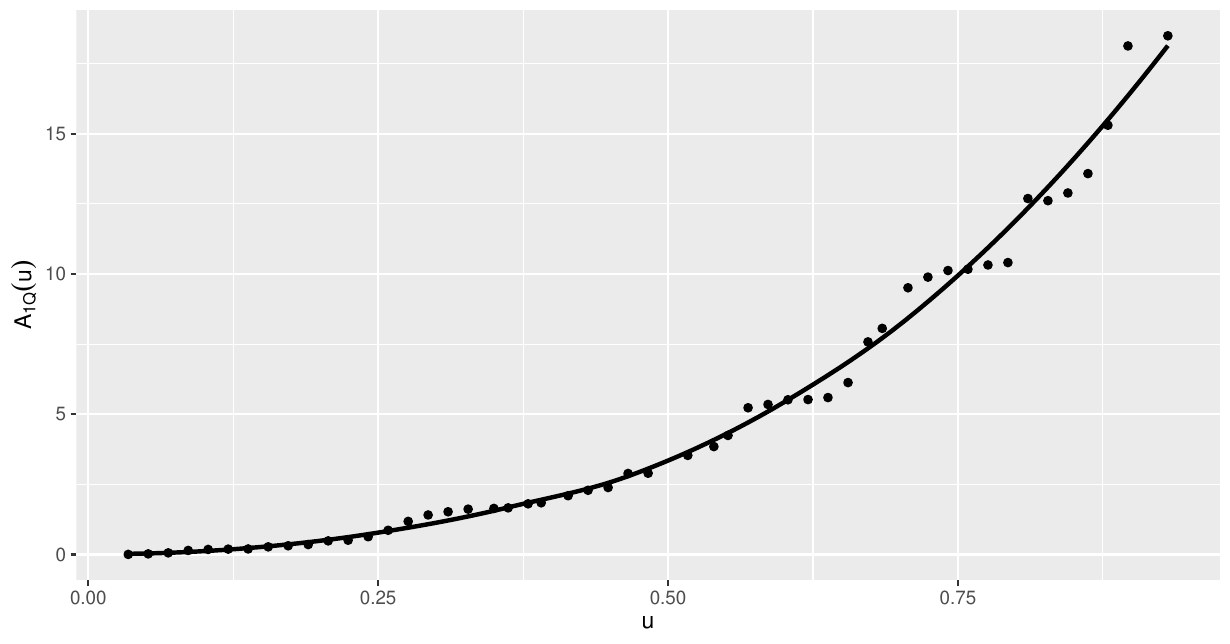}
	\caption{Sen index for the poor}
	\label{fig:uA1Qu}
\end{figure}
where we have assumed that $\log u$ ia evaluated at some point $u$ near the origin to give $r$. Notice that the quantile function \eqref{5.1} does not have a closed form distribution function to calculate the conventional indices necessitating the need for quantile-based poverty measures for further analysis. The other poverty measures are the income gap ratio,

$$
\begin{aligned}
I_Q(u) & =\frac{1}{u Q(u)} \int_0^u p q(p) d p=\frac{\mu_Q(u)}{Q(u)} \\
& =\frac{\alpha+\beta u}{\alpha \log u+2 \beta u+r}
\end{aligned}
$$
and the poverty gap ratio,

$$
\begin{aligned}
A_1(u) & =\frac{1}{Q(u)} \int_0^u p q(p) d p=u I_Q(u) \\
& =\frac{u(\alpha+\beta u)}{\alpha \log u+2 \beta u+r}
\end{aligned}
$$
and the Gini index for the poor, obtained from \eqref{3.11} is
$$
\begin{aligned}
& G_Q(u)=u-\frac{2}{Q(u)} \int_0^u(\alpha \log p+2 \beta p+r) d p+\frac{2}{Q(u)} \int_0^u\left(\alpha p \log p+2 \beta p^2+r p\right) d p \\
& =u-2 \frac{\alpha u(\log u-1)+\beta u^2+r u}{\alpha \log u+2 \beta u+r}+\frac{2\left(\alpha \frac{u^2}{2}-\frac{u^2}{4}+\frac{2 \beta u^3}{3}+\frac{r u^2}{2}\right)}{\alpha \log u+2 \beta u+r}.
\end{aligned}
$$
%We propose

%$$
%\bar{I}_{Q_i}(u_j)=\sum_{i=1}^{j-1} \frac{i\left(X_{i+1: n} -X_{i :n} \right)}{X_{j: n} }
%$$
%as the estimator of $I_Q$ so that $I_Q(u_j)=Q(u_j)I_Q(u_j)$ in agreement with the population relation $\mu_Q(u)=Q(u) I_Q(u)$. 
%The estimate of $G_Q$ is 

\begin{table}[H]
	\begin{center}
		\begin{tabular}{cccccc}
			$u$ & $I_Q(u)$ & $u$& $I_Q(u)$ & $u$& $I_Q(u)$ \\ \hline
			0.0345 & 0.0005 & 0.1207 & 0.0316 & 0.2069 & 0.0439 \\
			0.0517 & 0.0095 & 0.1379 & 0.0278 & 0.2241 & 0.0422 \\
			0.0690 & 0.0184 & 0.1552 & 0.0338 & 0.2414 & 0485 \\
			0.0862 & 0.0355 & 0.1724 & 0.0343 & 0.2586 & 0.0615 \\
			0.1034 & 0.0355 & 0.1897 & 0.0352 & 0.2759 & 0785 \\
			0.2931 & 0.0879 & 0.5690 & 0.1687 & 0.8276 & 0.2742 \\
			0.3103 & 0.0895 & 0.5862 & 0.1671& 0.8448 & 0.2742 \\
			0.3276 & 0.0897 & 0.6034 & 0.1671 & 0.8621 & 0.2829 \\
			0.3498 & 0.0863 & 0.6207 & 0.1624 & 0.8793 & 0.3123 \\
			0.3621 & 0.0830 & 0.6379 & 0.1598 & 0.8966 & 0.3626 \\
			0.3793 & 0.0859 & 0.6552 & 0.1702 & 0.9310 & 0.3626 \\
			0.3906 & 0.0837 & 0.6724 & 0.2048 & &  \\
			0.4138 & 0.0910 & 0.6847 & 0.2121 &  &  \\
			0.4310 & 0.0953 & 0.7069 & 0.2438 &  &  \\
			0.4483 & 0.0954 & 0.7241 & 0.2472 &  &  \\
			0.4655 & 0.1110 & 0.7414 & 0.2469 &  &  \\
			0.4828 & 0.1072 & 0.7586 & 0.2421 &  &  \\
			0.5170 & 0.1261 & 0.7759 & 0.2400 &  &  \\
			0.5395 & 0.1326 & 0.7931 & 0.2365 &  &  \\
			0.5517 & 0.1414 & 0.8103 & 0.2820 &  &  \\ \hline
		\end{tabular}
	\end{center}
	\caption{Estimates of income gap ratio}
	\label{table6}
\end{table}
%\begin{figure}[H]
%	\centering
%	\includegraphics[width=0.7\linewidth]{uIQu.pdf}
%	\caption{Estimates of income gap ratio.}
%	\label{fig:uIQ}
%\end{figure}
Similarly $G_Q(u)$ is estimated by plugging in $\bar{q}$ and $\bar{\mu}_Q$, to get

$$
G_Q\left(u_j\right)=\frac{2}{j^2 \bar{\mu}_Q}(j) \sum_{i=1}^{j-1} i^2\left(X_{i+1: n}-X_{i: n}\right).
$$

%The corresponding values can be seen in Table \ref{table5}. Finally we have the Sen index in Table 5. Graphs of the functions $A_1(u)$, $I_Q(u)$ and $S(u)$ seen from Figures \ref{fig:uGQu} - \ref{fig:uIQ} lead us to the following conclusions.\\
%\par As a by-product of this work, we can use our result in carrying out a goodness of fit test  in case black box modelling is adopted for determining the income distribution. Assume that the model with quantile function $Q(u)$ is tentatively chosen as appropriate for the data. Then from the expression for $I_Q(u)$ of the distribution, we calculate $I_Q\left(\frac{j}{n}\right)$ for $j=1,2, \ldots n$. When the sample is from the distributes $Q$, the values $I_Q\left(\frac{j}{n}\right)$ and $\bar{I}\left(u_j\right)$ should be approximately same. Accordingly the points $\left(\vec{I}\left(u_j\right), I_Q\left(\frac{j}{n}\right)\right)$ should fall along a straight line. When the points show departure from linearity the data does not give evidence in support of $Q$. Thus we have a test of model adequacy using the income gap ratio. Although we have obtained the income distribution using a characterization property of $\mu_Q(u)$, we see from Figure \ref{fig:uA1Qu} that the distribution we have assumed gives a reasonable fit to the income data.

We have presented in Tables \ref{table4}, \ref{table5}, \ref{table6}, the estimated values of IGR, Gini index and Sen index respectively based on the sample values.
\section*{Data Availability Statements}
Authors can confirm that all relevant data are included in the article and/or its supplementary information files.

\section*{Conflict of interest statement}
On behalf of all authors, the corresponding author states that there is no conflict of interest.
\bibliographystyle{apalike}
\bibliography{myref} 

\begin{thebibliography}{}

\bibitem[Aggarwal, 1984]{aggarwal1984optimum}
Aggarwal, V. (1984).
\newblock On optimum aggregation of income distribution data.
\newblock {\em Sankhy{\=a}: The Indian Journal of Statistics, Series B}, 46:343--355.

\bibitem[Atkinson, 1987]{atkinson1987measurement}
Atkinson, A.~B. (1987).
\newblock On the measurement of poverty.
\newblock {\em Econometrica: Journal of the Econometric Society}, 55:749--764.

\bibitem[Chakravarty, 1983]{chakravarty1983ethically}
Chakravarty, S.~R. (1983).
\newblock Ethically flexible measures of poverty.
\newblock {\em The Canadian Journal of Economics}, 16:74--85.

\bibitem[Chakravarty, 2019]{chakravarty2019shorrocks}
Chakravarty, S.~R. (2019).
\newblock On {S}horrocks’ reinvestigation of the {S}en poverty index.
\newblock {\em Poverty, Social Exclusion and Stochastic Dominance}, pages 27--29.

\bibitem[Chotikapanich, 1993]{chotikapanich1993comparison}
Chotikapanich, D. (1993).
\newblock A comparison of alternative functional forms for the {L}orenz curve.
\newblock {\em Economics Letters}, 41:129--138.

\bibitem[Clark et~al., 1981]{clark1981indices}
Clark, S., Hemming, R., and Ulph, D. (1981).
\newblock On indices for the measurement of poverty.
\newblock {\em The Economic Journal}, 91:515--526.

\bibitem[Foster et~al., 1984]{foster1984class}
Foster, J., Greer, J., and Thorbecke, E. (1984).
\newblock A class of decomposable poverty measures.
\newblock {\em Econometrica: Journal of the Econometric Society}, 52:761--766.

\bibitem[Gilchrist, 2000]{gilchrist2000statistical}
Gilchrist, W. (2000).
\newblock {\em Statistical Modelling with Quantile Functions}.
\newblock Chapman and Hall/CRC, London, UK.

\bibitem[Gupta, 1984]{gupta1984functional}
Gupta, M.~R. (1984).
\newblock Functional form for estimating the {L}orenz curve.
\newblock {\em Econometrica}, 52:1313--1314.

\bibitem[Hagenaars, 1987]{hagenaars1987class}
Hagenaars, A. (1987).
\newblock A class of poverty indices.
\newblock {\em International Economic Review}, 28:583--607.

\bibitem[Hagenaars, 2017]{hagenaars2017definition}
Hagenaars, A.~J. (2017).
\newblock The definition and measurement of poverty.
\newblock In {\em Economic Inequality and Poverty}, pages 134--156. Routledge.

\bibitem[Haridas et~al., 2008]{haridas2008modelling}
Haridas, H.~N., Nair, N.~U., and Nair, K. R.~M. (2008).
\newblock Modelling income using the generalised lambda distribution.
\newblock {\em Journal of Income Distribution}, 17(2):37--51.

\bibitem[Hosking, 1994]{hosking1994four}
Hosking, J.~R. (1994).
\newblock The four-parameter kappa distribution.
\newblock {\em IBM Journal of Research and Development}, 38:251--258.

\bibitem[Houghton, 1978]{houghton1978birth}
Houghton, J.~C. (1978).
\newblock Birth of a parent: The wakeby distribution for modeling flood flows.
\newblock {\em Water Resources Research}, 14:1105--1109.

\bibitem[Kakwani, 1980]{kakwani1980class}
Kakwani, N. (1980).
\newblock On a class of poverty measures.
\newblock {\em Econometrica: Journal of the Econometric Society}, 48:437--446.

\bibitem[Kakwani and Podder, 1973]{kakwani1973estimation}
Kakwani, N.~C. and Podder, N. (1973).
\newblock On the estimation of {L}orenz curves from grouped observations.
\newblock {\em International Economic Review}, 14:278--292.

\bibitem[Karian and Dudewicz, 2000]{karian2000fitting}
Karian, Z.~A. and Dudewicz, E.~J. (2000).
\newblock {\em Fitting Statistical Distributions: The Generalized Lambda Distribution and Generalized Bootstrap Methods}.
\newblock Chapman and Hall/CRC.

\bibitem[Mudholkar and Kollia, 1994]{mudholkar1994generalized}
Mudholkar, G.~S. and Kollia, G.~D. (1994).
\newblock Generalized {W}eibull family: a structural analysis.
\newblock {\em Communications in Statistics-Theory and Methods}, 23:1149--1171.

\bibitem[Nair et~al., 2008]{nair2008some}
Nair, N.~U., Nair, K.~M., and Haridas, H.~N. (2008).
\newblock Some properties of income gap ratio and truncated gini coefficient.
\newblock {\em Calcutta Statistical Association Bulletin}, 60:245--254.

\bibitem[Nair et~al., 2013]{nair2013quantile}
Nair, N.~U., Sankaran, P.~G., and Balakrishnan, N. (2013).
\newblock {\em Quantile-Based Reliability Analysis}.
\newblock Springer, US.

\bibitem[Nair and Vineshkumar, 2022]{nair2022cumulative}
Nair, N.~U. and Vineshkumar, B. (2022).
\newblock Cumulative entropy and income analysis.
\newblock {\em Stochastics and Quality Control}, 37:165--179.

\bibitem[Ortega et~al., 1991]{ortega1991new}
Ortega, P., Martin, G., Fernandez, A., Ladoux, M., and Garcia, A. (1991).
\newblock A new functional form for estimating {L}orenz curves.
\newblock {\em Review of Income and Wealth}, 37:447--452.

\bibitem[Ramberg and Schmeiser, 1972]{ramberg1972approximate}
Ramberg, J.~S. and Schmeiser, B.~W. (1972).
\newblock An approximate method for generating symmetric random variables.
\newblock {\em Communications of the ACM}, 15:987--990.

\bibitem[Rohde, 2009]{rohde2009alternative}
Rohde, N. (2009).
\newblock An alternative functional form for estimating the {L}orenz curve.
\newblock {\em Economics Letters}, 41:21--29.

\bibitem[Sen, 1976]{sen1976poverty}
Sen, A. (1976).
\newblock Poverty: an ordinal approach to measurement.
\newblock {\em Econometrica: Journal of the Econometric Society}, 44:219--231.

\bibitem[Sen, 1986]{sen1986gini}
Sen, P.~K. (1986).
\newblock The gini coefficient and poverty indices: some reconciliations.
\newblock {\em Journal of the American Statistical Association}, 81:1050--1057.

\bibitem[Shorrocks, 1995]{shorrocks1995revisiting}
Shorrocks, A.~F. (1995).
\newblock Revisiting the {S}en poverty index.
\newblock {\em Econometrica: Journal of the Econometric Society}, 63:1225--1230.

\bibitem[Takayama, 1979]{takayama1979poverty}
Takayama, N. (1979).
\newblock Poverty, income inequality, and their measures: Professor {S}en's axiomatic approach reconsidered.
\newblock {\em Econometrica: Journal of the Econometric Society}, 47:747--759.

\bibitem[Tarsitano et~al., 2004]{tarsitano2004fitting}
Tarsitano, A. et~al. (2004).
\newblock Fitting the generalized lambda distribution to income data.
\newblock In {\em COMPSTAT’2004 Symposium}, pages 1861--1867. Springer.

\bibitem[Tarsitano et~al., 2006]{tarsitano2006new}
Tarsitano, A. et~al. (2006).
\newblock A new {Q-Q} plot and its application to income data.
\newblock {\em Statistica \& Applicazioni}, 4.

\bibitem[Thon, 1979]{thon1979measuring}
Thon, D. (1979).
\newblock On measuring poverty.
\newblock {\em Review of Income and Wealth}, 25:429--439.

\bibitem[Thon, 1983]{thon1983poverty}
Thon, D. (1983).
\newblock A poverty measure.
\newblock {\em The Indian Economic Journal}, 30:285--307.

\bibitem[Watts, 1969]{watts1969economic}
Watts, H.~W. (1969).
\newblock An economic definition of poverty.
\newblock {\em In Moynitian, D Ed. On Understanding Poverty}, pages 316--329.

\bibitem[Xu, 2020]{xu2020sen}
Xu, K. (2020).
\newblock The {S}en-{S}horrocks-{T}hon index of poverty intensity.
\newblock In {\em Encyclopedia of Quality of Life and Well-Being Research}, pages 1--3. Springer.

\bibitem[Yang, 2017]{yang2017relationship}
Yang, L. (2017).
\newblock The relationship between poverty and inequality: Concepts and measurements.
\newblock Working paper, Centre for Analysis of Social Exclusion, London School of Economics.

\end{thebibliography}

\end{document}